\documentclass[11pt,leqno]{article}

\usepackage{amssymb,amsmath,amsthm}

\newcommand{\bb}[1]{\mathbb{#1}}
\newcommand{\cl}[1]{\mathcal{#1}}

\theoremstyle{plain}

\newtheorem{thm}{Theorem}

\newtheorem{lem}[thm]{Lemma}
\newtheorem{cor}[thm]{Corollary}

\theoremstyle{remark}
\newtheorem*{rk}{Remark}
\newtheorem{rem}[thm]{Remark}

\theoremstyle{definition}

\begin{document}

\title{Real interpolation  and transposition of certain function spaces}

\author{by\\
Gilles Pisier\\
Texas A\&M University\\
College Station, TX 77843, U. S. A.\\
and\\
Universit\'e Paris VI\\
Equipe d'Analyse, Case 186, 75252\\
Paris Cedex 05, France}

\date{}

\maketitle

\begin{abstract} 
Our starting point is a lemma due to Varopoulos.
We give
a different proof of a generalized form this lemma, that yields an equivalent description of the $K$-functional
for the interpolation couple $(X_0,X_1)$ where
$X_0=L_{p_0,\infty}(\mu_1; L_q(\mu_2))$
and $X_1=L_{p_1,\infty}(\mu_2; L_q(\mu_1))$
where $0<q<p_0,p_1\le \infty$ and $(\Omega_1,\mu_1), (\Omega_2,\mu_2)$ are arbitrary measure spaces. When $q=1$, this implies that the space $(X_0,X_1)_{\theta,\infty}$ ($0<\theta<1$) can be identified
with a certain  space of operators. 
We also give an extension of the Varopoulos Lemma to pairs (or finite families) of conditional expectations that seems of independent interest.
The present paper is motivated by non-commutative applications
that we choose to publish separately.
\end{abstract}

Motivated by certain non-commutative analogues of a Lemma due to Varopoulos \cite{Va}, we noticed an extension of his lemma that was overlooked in \cite{HP}.
 The main result is a dual characterization of the functions in the space
\begin{equation}\label{real-eq1}
L_{p_1,\infty} (\mu_1; L_q(\mu_2)) + L_{p_2,\infty}(\mu_2;L_q(\mu_1))
\end{equation}
when $0<q<p_1,p_2\le \infty$ and $(\Omega_1,\mu_1), (\Omega_2,\mu_2)$ are arbitrary measure spaces.

The equivalent condition for a (measurable) function $f\colon \ \Omega_1\times\Omega_2\to {\bb R}$ to belong to this space is
\begin{equation}\label{real-eq2}
\sup \int_{E_1\times E_2} |f| \ d\mu_1 d\mu_2 \left(\mu_1(E_1)^{\frac1q -\frac1{p_1}} + \mu_2(E_2)^{\frac1q-\frac1{p_2}}\right)^{-1} < \infty
\end{equation}
where the sup runs over all possible (measurable) subsets $E_j\subset\Omega_j$ $(j=1,2)$.\\
In \cite{HP} only the case $p_1=p_2$ is considered and the proof does not seem to extend further.
This result extends to functions $f(\omega_1,\omega_2,\ldots,\omega_n)$ of any number of variables. The relevant space is then
\[
\sum\nolimits^n_{j=1} L_{p_j,\infty}(\mu_j; L_q(\widehat\mu_j))
\]
where $\widehat\mu_j = \mu_1 \times\cdots\times \mu_{j-1} \times \mu_{j+1} \times\cdots\times \mu_n$.

Our main result can of course be formulated as a two-sided inequality expressing the equivalence of the norms appearing in \eqref{real-eq1} and \eqref{real-eq2}.
When we write this inequality for the pair $(\mu_1,t\mu_2)$ with $t>0$, we obtain an equivalent form of the $K$-functional for the pair of spaces composing the sum in \eqref{real-eq1}.

Using this expression for the $K$-functional, one derives a description for the real interpolation space
\[
(X_1,X_2)_{\theta,\infty}
\]
when $X_j = L_{p_j,\infty}(\mu_j; L_q(\widehat\mu_j))$, $j=1,2$. The condition we find for this space is particularly striking, it is simply
\[
\sup_{E_j\subset\Omega_j} \left(\int_{E_1\times E_2} |f|^q \ d\mu_1 d\mu_2 \right)^{1/q}(\mu(E_1)^{\alpha_1(1-\theta)} \mu(E_2)^{\alpha_2\theta})^{-1}.
\]
where $\alpha_j = \frac1q - \frac1{p_j}$.

Let $(\Omega,\mu)$ be a measure space. Recall that the ``weak $L_p$'' space $L_{p,\infty}(\mu)$ is formed of all measurable functions $f\colon \ \Omega\to {\bb R}$ such that
\[
\|f\|_{p,\infty} = \sup_{c>0} (c^p\mu\{|f|>c\})^{1/p} < \infty.
\]
When $p>1$, the quasi-norm $\|\cdot\|_{p,\infty}$ is equivalent to the following norm
\begin{equation}\label{99}
\|f\|_{[p,\infty]} = \sup\left\{\int_E |f| \frac{d\mu}{\mu(E)^{1/p'}}\ \Big|\ E\subset \Omega\right\}.
\end{equation}
When $p>1$,  $L_{p,\infty}(\mu)$ is the dual of the ``Lorentz space'' $L_{p',1}(\mu)$ that can be defined
(see e.g. \cite{BL}) as formed of those $f$ such that
\begin{equation}\label{100}
[f]_{p,1} = \int^\infty_0 \mu\{|f|>c\}^{1/p} \ dc < \infty.
\end{equation}

For a Banach space valued function $f\colon \ \Omega\to B$, we set (for $p>1$)
\begin{equation}\label{101}
\|f\|_{L_{p,\infty}(\mu ; B)} = \sup_{E\subset\Omega} \int_E\|f\| d\mu \ \mu(E)^{-1/p},
\end{equation}
and we denote by ${L_{p,\infty}(\mu ; B)}$ the space of functions in $L_1(\mu ; B)$ (in Bochner's sense)
for which the latter norm is finite.

\begin{rem}\label{rk-lo}
Let $(\Omega,\mu)$ be any measure space. Let $f\colon \ \Omega\to {\bb R}$ be any measurable function. Note that
\begin{equation}\label{re-eq11}
|f| = \int^\infty_0 1_{\{|f|>c\}} dc = \int^\infty_0 \mu\{|f| >c\}^{1/p} \varphi_c \ dc
\end{equation}
where $\varphi_c = \mu\{|f|>c\}^{-1/p} 1_{\{|f|>c\}}$. Let ${\cl S}$ be the space of integrable step functions and let $T\colon \ {\cl S}\to B$ be a linear map into a Banach space. Then $T$ extends boundedly to $L_{p,1}(\mu)$ iff there is a constant $c$ such that for any measurable subset $E\subset\Omega$ and any $g$ in $L_\infty(\mu)$ we have
\[
\|T(g1_E)\|_B \le c\mu(E)^{1/p}.
\]
This well known fact can be derived easily from \eqref{100}. Indeed, if $g=f|f|^{-1}$, we deduce from \eqref{re-eq11} that if $\|f\|_{p,1}\le 1$  then $f$ lies in the closed convex hull of the set $\{g\varphi_c\mid c>0\}$. 
\end{rem}

We start by stating
the Varopoulos lemma. In   \cite{Va}   he proved it for $\Omega_1=\Omega_2 = [1,\ldots, n]$ equipped with counting measure. Later on, some generalizations were given in \cite{HP}. 

\begin{lem}\label{varo}
Consider $f\colon \ \Omega_1\times\Omega_2 \to {\bb R}$ measurable. Let $X_1=L_\infty(\mu_1; L_1(\mu_2))$ and $X_2= L_\infty(\mu_2;L_1(\mu_1))$. Consider the following two properties:
\begin{itemize}
\item[\rm (i)] There are $f_1\in X_1, f_2\in X_2$ such that
\[
f = f_1+f_2 \quad \text{and}\quad \|f_1\|_{X_1} + \|f_2\|_{X_2}\le 1.
\]
\item[\rm (ii)] For any pair of measurable subsets $E_1\subset \Omega_1, E_2\subset\Omega_2$, we have
\[
\int_{E_1\times E_2} |f| \ d\mu_1 d\mu_2 \le \mu_1(E_1) + \mu_2(E_2).
\]
Then (i) $\Rightarrow$ (ii) and (ii) implies (i) for the function $f/2$.
\end{itemize}
\end{lem}

The Varopoulos proof, and that of \cite{HP} that mimics it, focus on the case when $f$ is an $n\times n$ matrix and use induction on $n$.

We found a completely different very direct (but dual) proof as follows: 
\begin{proof}
 That (i) $\Rightarrow$ (ii) is obvious. Conversely,
assume (ii). To show that $f/2$ satisfies (i), it suffices by duality to show that
\begin{equation}\label{real-eq3}
\left|\int fg\right| \le 2
\end{equation}
for any $g\colon \ \Omega_1\times\Omega_2$ such that $\max\{\|g\|_{X^*_1}, \|g\|_{X^*_2}\}\le 1$,  equivalently we may restrict to $g$ such that
\[
\max\left\{\int \|g\|_{L_\infty(\mu_2)} d\mu_1, \int \|g\|_{L_\infty(\mu_1)} d\mu_2\right\} \le 1.
\]
Let
\[
\alpha(\omega_1) = \|g(\omega_1,\omega_2)\|_{L_\infty(d\mu(\omega_2))} \quad \text{and}\quad \beta(\omega_2) = \|g(\omega_1,\omega_2)\|_{L_\infty(d\mu(\omega_1))}.
\]
We set $\varphi = \alpha\wedge\beta$ (where $(\alpha\wedge\beta)(t) \overset{\text{def}}{=} \min(\alpha(t), \beta(t))$). We have
\[
g = \varphi\cdot \widetilde g
\]
where $\|\widetilde g\|_{L_\infty(\mu_1\times \mu_2)} \le 1$. We now use $\alpha\wedge \beta = \int^\infty_0 1_{\{\alpha\wedge \beta>c\}} dc$, this gives us
\[
\alpha\wedge\beta = \int^\infty_0 1_{E_c\times F_c} \ dc
\]
where $E_c = \{\alpha>c\}$ and $F_c = \{\beta>c\}$. We can rewrite this as
\[
\alpha\wedge\beta = \int^\infty_0 (\mu_1(E_c) + \mu_2(F_c) )\varphi_c\ dc
\]
where $\varphi_c = (\mu_1(E_c) + \mu_2(F_c))^{-1} 1_{E_c\times F_c}$. Then, observing that $\int^\infty_0 \mu_1(E_c) + \mu_2(E_c)dc=$ $\int \alpha d\mu_1 + \int \beta \ d\mu_2 \le 2$, we find
\[
\int |fg| \ d\mu_1 d\mu_2 = \left|\int |f(\alpha\wedge\beta)\widetilde g|\right| \le 2 \sup_c \int|f|\varphi_c \ d\mu_1 d\mu_2
\]
but (ii) implies $\int |f|\varphi_c\le 1$ so we obtain \eqref{real-eq3}. 
\end{proof}
\begin{rk} We could not find a modified formulation
avoiding the presence of some  extra factor (here equal to 2)  spoiling the equivalence in Lemma \ref{varo}, but this might exist.
\end{rk}
Our extension of the Varopoulos result is based on the following

\begin{lem}\label{lem100}
Let $1<p_1,p_2<\infty$. Let $f\colon \ \Omega_1\times\Omega_2\to {\bb R}$ be measurable such that
\[
\sup_{E_j\subset \Omega_j}\left\{\int_{E_1\times E_2} |f| \ d\mu_1 d\mu_2 (\mu_1(E_1)^{1/p'_1} + \mu_2(E_2)^{1/p'_2})^{-1}\right\} \le 1.
\]
Then there is a decomposition
\[
f=f_1+f_2
\]
with $\|f_1\|_{L_{p_1,\infty}(\mu_1;L_1(\mu_2))} + \|f_2\|_{L_{p_2,\infty}(\mu_2; L_1(\mu_1))}\le 2$. \end{lem}

\begin{proof}
We proceed exactly as in the above proof. By duality it suffices to show
\[
\int |fg|\le 2
\]
for any $g\colon \ \Omega_1\times\Omega_2$ such that
\[
\max\{\|g\|_{L_{p'_1,1}(\mu_1,L_1(\mu_2))}, \|g\|_{L_{p'_2,1}(\mu_2,L_1(\mu_1))}\} \le 1.
\]
Let $\alpha$ and $\beta$ be as before.

We now have
\[
\max\left\{\int^\infty_0 \mu_1(\alpha>c)^{1/p'_1} dc, \int^\infty_0 \mu_2\{\beta>c\}^{1/p'_2} dc\right\}\le 1.
\]
We can thus write
\[
\alpha\wedge\beta = \int^\infty_0 (\mu_1(E_c)^{1/p'_1} + \mu_2(F_c)^{1/p'_2})\psi_c \ dc
\]
where $\psi_c = (\mu_1(E_c)^{1/p'_1} + \mu_2(F_c)^{1/p'_2})^{-1} 1_{E_c\times F_c}$. Since our assumption implies
\[
\int |f|\psi_c \ d\mu_1 d\mu_2 \le 1
\]
we conclude as before that
\[
\int |fg| d\mu_1 d\mu_2 \le 2 \sup_c \int |f|\psi_c \le 2.\eqno \qed
\]
\renewcommand{\qed}{}\end{proof}

\begin{rem}\label{rem-q}
Let $0<q<\infty$. The preceding Lemma applied to $|f|^{q}$ gives us a sufficient condition for $f$ to decompose as a sum $f=f_1+f_2$ with
\[
f_1\in L_{qp_1,\infty} (\mu_1; L_q(\mu_2)), f_2\in L_{qp_2,\infty}(\mu_2;L_q(\mu_1)).
\]
\end{rem}

We now give applications to the real interpolation method. We refer to \cite{BL} for all undefined terms.
We just recall that if $(A_0,A_1)$ is a compatible couple of Banach spaces,
then for any $x\in A_0+A_1$ the $K$-functional is defined by
$$\forall t>0\qquad K_t(x;A_0,A_1)= \inf
\big({\|x_0\|_{A_0}+t\|x_1\|_{A_1}\ | \ x=x_0+x_1,x_0\in
A_0,x_1\in A_1}). $$
Recall that the (``real" or ``Lions-Peetre" interpolation) space
$(A_0,A_1)_{\theta,q}$ is defined, for $0<\theta<1$ and $1\le q\le \infty$, as the space of all $x$
in $A_0+A_1$ such that $\|x\|_{\theta,q} <\infty$ where
 $$\|x\|_{\theta,q} =(\int{(t^{-\theta}K_t(x,A_0,A_1))^{q}
dt/t})^{1/q} ,$$
with the usual convention when $q=\infty$.

\begin{thm}\label{thm8}
Let $f\colon \ \Omega_1\times\Omega_2\to {\bb R}$ be measurable. Let $0<q<p_1,p_2\le \infty$. For simplicity of notation, let
\[
K_t(f) =K_t(f; L_{p_1,\infty}(\mu_1; L_q(\mu_2)), L_{p_2,\infty}(\mu_2; L_q(\mu_1)).
\]
Then there are positive constants $c,C$ depending only on $q,p_1,p_2$ such that
\begin{equation}
ck_t(f) \le K_t(f) \le Ck_t(f)\tag*{$\forall t>0$}
\end{equation}
where
\begin{equation}\label{re-eq10}
k_t(f) = \sup\left\{\left(\int_{E_1\times E_2} |f|^q\ d\mu_1 d\mu_2\right)^{1/q} (\mu _1(E_1)^{\frac1q-\frac1{p_1}} + t^{-1}\mu(E_2)^{\frac1q-\frac1{p_2}})^{-1}\right\}.
\end{equation}
\end{thm}

\begin{proof}
Assume first that $q=1$. Let $X_1=L_{p_1,\infty}(\mu_1; L_1(\mu_2))$, $X_2 = L_{p_2,\infty}(\mu_2; L_1(\mu_1))$. The Lemma applied to the pair $(\mu_1,t\mu_2)$ in place of the pair $(\mu_1,\mu_2)$ and then dividing the result by $t$ shows that $t^{-1} \inf\limits_{f=f_1+f_2}\{t\|f_1\|_{X_1} + t^{1/p_2}\|f_2\|_{X_2}\}$ is equivalent to 
\[
\sup_{E_1,E_2} \left\{\int_{E_1\times E_2} |f|\ d\mu_1 d\mu_2(\mu_1(E_1)^{1/p'_1} + t^{1/p'_2}\mu(E_2)^{1/p'_2})^{-1}\right\}.
\]
Thus we find that this last expression is equivalent to $K_s(f)$ for $s=t^{\frac1{p_2}-1}=t^{-\frac1{p'_2}}$. If we now change variable from $t$ to $s$ we find $K_s(f)$ equivalent to
\[
\sup_{E_1,E_2} \left\{\iint_{E_1\times E_2} |f|\  d\mu_1 d\mu_2 (\mu_1(E_1)^{1/p'_1} + s^{-1}\mu_2(E_2)^{1/p'_2})\right\}.
\]
If we now apply the case $q=1$ that we just verified with $|f|^q$ in place of $|f|$, 
and use Remark \ref{rem-q}, we find an equivalent for $K_s(|f|^q; X_1,X_2)$. The result then follows from elementary calculations: since
\[
(K_s(|f|^q; X_1,X_2))^{1/q} \simeq K_{s^{1/q}}(f; L_{qp_1,\infty}(\mu_1; L_q(\mu_2)), L_{qp_2,\infty}(\mu_2; L_q(\mu_1))),
\]
we find that $K_{s^{1/q}}(f; L_{qp_1,\infty}(\mu_1; L_q(\mu_2)), L_{qp_2,\infty}(\mu_2; L_q(\mu_1)))$ is equivalent to
\[
\sup_{E_1,E_2} \left\{ \left(\int_{E_1\times E_2} |f|^q\ d\mu_1 d\mu_2\right)^{1/q}(\mu_1(E_1)^{1/qp'_1} + s^{-1/q}\mu(E_2)^{1/qp'_2})^{-1}\right\}.
\]
The   final adjustment consists in replacing $(qp_1,qp_2)$ by $(p_1,p_2)$
and ${s^{1/q}}$ by $t$, we then find \eqref{re-eq10}.
\end{proof}

\begin{cor}\label{cor9}
Let $f$ be as in Theorem~\ref{thm8}. Then $f\in (L_{p_1,\infty}(\mu_1; L_q(\mu_2)), L_{p_2,\infty}(\mu_2; L_q(\mu_1)))_{\theta,\infty}$ iff
\[
\sup_{E_1,E_2\subset\Omega} \left(\int_{E_1\times E_2} |f|^q\ d\mu_1 d\mu_2\right)^{\frac1q} \cdot (\mu_1(E_1)^{(1-\theta)\alpha_1} \mu(E_2)^{\theta\alpha_2})^{-1} < \infty
\]
where $\alpha_j =  \frac1q-\frac1{p_j}$ $(j=1,2)$. The corresponding norms are equivalent. 
When $q=1$, this condition means that 
the operator admitting  $|f|$ as its kernel, namely the operator  
\[
T_{|f|}\colon \ g\to \int|f(x,y)| g(y) d\mu_2(y)
\]
  is bounded from $L_{r,1}(\mu_2)$ to $L_{s,\infty}(\mu_1)$ where $\frac1r = \theta\alpha_2 = \frac\theta{p'_2}$ and $\frac1{s'} = 1 - \frac1s = \frac{(1-\theta)}{p'_1}$. 
\end{cor}

\begin{proof}
The first part is clear using the definition of the norm in $(~~,~~)_{\theta,\infty}$ and the identity:
\begin{equation}\label{re-eq10a}
\forall a_0,a_1>0\qquad\qquad a^{1-\theta}_0a^\theta_1 = \inf_{t>0} \{(1-\theta) a_0t^\theta + \theta a_1t^{\theta-1}\}.
\end{equation}
The second part follows from Remark \ref{rk-lo}
\end{proof}

\begin{rem}\label{rem19} Let us say that an operator $T$
from a Lorentz space $X$ to another one $Y$  is regular if it is a linear combination
of positive (i.e. preserving positivity) bounded operators from $X$ to $Y$. In the complex case this means that
$T$ can be decomposed as $T=T_1-T_2 +i(T_3-T_4 ) $ with all $T_j$'s positive  and bounded   from $X$ to $Y$. (In the real case the imaginary part can be ignored).
We will say that $T$ is a   kernel operator
if it is defined using a  scalar kernel $f$ that is measurable on $\Omega_1\times \Omega_2$. It is easy to check that a kernel operator
$T$ is regular from $L_{r,1}(\mu_2)$ to $L_{s,\infty}(\mu_1)$ iff its kernel $f$ is such that
\[
T_{|f|}\colon \ g\to \int|f(x,y)| g(y) d\mu_2(y)
\]
  is bounded from $L_{r,1}(\mu_2)$ to $L_{s,\infty}(\mu_1)$.
  In the real case, this means equivalently that there is a positive operator $S$
  such that $\pm T\le S$ that is bounded from $L_{r,1}(\mu_2)$ to $L_{s,\infty}(\mu_1)$.
  It is worthwhile to observe that  
    $$\|T_{|f|}\colon\ L_\infty(\mu_2)\to L_{p_1,\infty}(\mu_1)\|=\|f\|_{L_{p_1,\infty}(\mu_1; L_1(\mu_2))}$$
    and  $$\|T_{|f|}\colon\ L_{p'_1,1}(\mu_2) \to L_1(\mu_1) \|=\|f\|_{L_{p_2,\infty}(\mu_2; L_1(\mu_1))}.$$
  Let $B_0$ (resp. $B_1$) be the space of regular kernel operators   $T \colon\ L_\infty(\mu_2)\to L_{p_1,\infty}(\mu_1)$ (resp. $T \colon\ L_{p'_1,1}(\mu_2) \to L_1(\mu_1) $).
  Then the preceding Corollary can be interpreted, when $q=1$, as the identification
  of $(B_0,B_1)_{\theta,\infty}$ with the space of regular kernel operators  $T \colon\ L_{r,1}(\mu_2) \to L_{s,\infty}(\mu_1) $, with $\frac 1r=\frac {1-\theta}{\infty}+\frac{\theta}{p_1'}$ and
  $ \frac 1s=\frac {1-\theta}{p_1}+\frac{\theta}{1}$.

\end{rem}
The extension of these results to functions $f$ of $n$ variables is immediate (note that the constant $2$ becomes $n$). We merely state the main point.

Consider measure spaces $(\Omega_j;\mu_j)$ $1\le j\le n$ and a measurable function $f\colon \ \Omega_1\times\cdots\times \Omega_n \to {\bb R}$.

Let $0 < q < p_1,\ldots, p_n\le \infty$. Let $Y_j = L_{p_j,\infty}(\mu_j, L_q(\widehat\mu_j))$. Then $f\in Y_1 +\cdots+ Y_n$ iff
\[
\sup_{E_j\subset\Omega_j} \left(\int_{E_1\times\cdots\times E_n} |f|^q d\mu_1\ldots d\mu_n\right)^{1/q} (\mu_1(E_1)^{\alpha_1} +\cdots+ \mu_n(E_n)^{\alpha_n})^{-1} < \infty
\]
where $\alpha_j = \frac1q-\frac1{p_j}$.

Replacing each $\mu_j$ by $ s_j \mu_j$ ($s_j>0$)  and then setting $t_j=s_j^{-\alpha_j}$
we find that there are constants $c,C>0$ such that 
the generalized $K$-functional
$$K(t_1,\cdots,t_n)=\inf\{ t_1\|x_1\|_{Y_1}+\cdots +t_n \|x_n\|_{Y_n}\mid x=x_1+\cdots+x_n\}$$
 is equivalent (with constants independent of $t=(t_j)$)
 to
 $$  \sup_{E_j\subset\Omega_j} \left(\int_{E_1\times\cdots\times E_n} |f|^q d\mu_1\ldots d\mu_n\right)^{1/q} (t_1^{-1}\mu_1(E_1)^{\alpha_1} +\cdots+ t_n^{-1}\mu_n(E_n)^{\alpha_n})^{-1} .$$
This gives a new example where the $K$-functionals considered in \cite{S}
(see also \cite{CP}) can be computed, at least  up to equivalence.

\begin{rem}\label{rem20}
The preceding Lemma \ref{lem100} can be reformulated as showing the following implication:\ If
\begin{equation}
\int_{E_1\times E_2} |f| \le \mu_1(E_1)^{1/p'_1} + \mu_2(E_2)^{1/p'_2}, \tag*{$\forall E_j\subset \Omega_j$}
\end{equation}
then there is a decomposition $f=f_1+f_2$ with $f_1,f_2$ such that
\begin{equation}
\int_{E_1\times E_2} |f| \le 2\mu_1(E_1)^{1/p'_1} \quad \text{and}\quad \int_{E_1\times E_2} |f_2| \le 2\mu_2(E_2)^{1/p'_2}. \tag*{$\forall E_j\subset \Omega_j$}
\end{equation}

Except for the factor 2, this resembles very much the kind of statements that are usually proved by the Hahn--Banach theorem, but we do not see how to prove it in this way.

More generally, the above simple minded argument extends
to the spaces  originally introduced by G.G. Lorentz
who denoted them by $\Lambda(\varphi)$ in \cite{Lo}. Here $\varphi$ is a non-negative
decreasing (meaning non-increasing) function on an interval of the real line
equipped with Lebesgue measure. We denote by $\Phi$ the primitive of $\varphi$ that vanishes at the origin
(so $\Phi$  is concave and increasing).
One can obtain a generalization
of the preceding decomposition to the case when
$x\mapsto x^{1/p_1'}$ and $x\mapsto x^{1/p_2'}$  are replaced by
two such functions $x\mapsto \Phi_1(x)$ and $x\mapsto \Phi_2(x)$. 
\end{rem}

There is a generalization of Lemma \ref{varo} to functions of not necessarily independent variables that seems to be of independent interest, as follows. Consider two conditional expectation operators ${\bb E}_j\colon \ L_1(\mu)\to L_1(\mu)$ $(j=1,2)$ on a measure space $(\Omega,{\cl A},\mu)$. For simplicity, we assume that $\mu(\Omega)=1$ but this is not really essential. Then let ${\cl B}_j\subset {\cl A}$ be the $\sigma$-subalgebra that is fixed by ${\bb E}_j$. Let $C_j$ be the space of measurable scalar valued functions $x$ on $(\Omega,\mu)$ such that
\begin{equation}\label{eq200}
\|x\|_{C_j} \overset{\text{def}}{=} \|{\bb E}_j(|x|)\|_\infty <\infty.
\end{equation}
We view $(C_1,C_2)$ as a compatible pair of Banach spaces. Consider $x\in C_1+C_2$ with $\|x\|_{C_1+C_2} \le 1$. Then a simple verification shows that for any pair of subsets $E_j\subset\Omega$ $(j=1,2)$, with $E_j$ assumed ${\cl B}_j$-measurable, we have
\begin{equation}\label{eq100}
\int_{E_1\cap E_2} |x| d\mu \le \mu(E_1) + \mu(E_1).
\end{equation}
Conversely, the above proof of Lemma \ref{varo} shows that \eqref{eq100} implies that
\[
\|x\|_{C_1+C_2}\le 2.
\]
Indeed, we may run the same duality argument:\ given $y$ such that both $|y|\le f_j$ $(j=1,2)$ with $f_j$, ${\cl B}_j$-measurable such that $\|f_j\|_1\le 1$. Then we have $y=(f_1\wedge f_2)\widehat y$ with $\|\widehat y\|_\infty \le 1$ and
\[
f_1\wedge f_2 = \int^\infty_0 1_{E^c_1\cap E^c_2} \ dc = \int^\infty_0 (\mu(E^c_1)+ \mu(E^c_2))\varphi_c \ dc
\]
with 
$E^c_j = \{f_j>c\}$
and $$\varphi_c = 1_{E^c_1\cap E^c_2}(\mu(E^c_1) +\mu(E^c_2))^{-1}.$$Thus we find that \eqref{eq100} implies $\int|xy| d\mu\le 2$ and we conclude by duality.

The same argument works for any number of conditional expectations. In terms of operators and kernels, the norm \eqref{eq200} can be described like this: we associate to $x\in C_j$ the operator $M_x$ of multiplication by $x$ on $L_1({\cl A},\mu)$. We then have
\[
\|x\|_{C_j} = \|M_x\colon \ L_1({\cl B}_j,\mu)\to L_1({\cl A},\mu)\|.
\]
We leave the extension of Lemma~\ref{lem100} to the reader.

In a separate paper \cite{P5}, we give non-commutative generalizations
of the preceding results.

\end{document}